\documentclass[12pt, reqno]{amsart}
\usepackage{amsmath}
\usepackage{amssymb}
\usepackage{epsfig}
\usepackage{graphicx}
\usepackage{color}
\usepackage{fullpage}
\usepackage{comment}
\definecolor{shadecolor}{gray}{0.875}
\usepackage{amscd}
\usepackage{stmaryrd}

\numberwithin{equation}{section}

\input xy
\xyoption{all}

\calclayout
\allowdisplaybreaks[3]

\theoremstyle{plain}
\newtheorem{prop}{Proposition}[section]

\newtheorem{theo}[prop]{Theorem}

\newtheorem{lemm}[prop]{Lemma}

\theoremstyle{definition}
\newtheorem{defi}[prop]{Definition}

\newtheorem{conj}[prop]{Conjecture}

\newtheorem{rema}[prop]{Remark}

\newtheorem{exam}[prop]{Example}

\makeatother
\makeatletter

\author{Sho Tanimoto}
\address{Department of Mathematics, Faculty of Science, Kumamoto University, Kurokami 2-39-1 Kumamoto 860-8555 Japan}
\address{Priority Organization for Innovation and Excellence, Kumamoto University}
\email{stanimoto@kumamoto-u.ac.jp}

\title[Geometric aspects of Manin's Conjecture]{Geometric aspects of Manin's Conjecture}

\begin{document}
\date{\today}

\begin{abstract}
This is a report of the author's talk at Kinosaki Algebraic Geometry Symposium 2018. We discuss some recent progress on the geometry of thin exceptional sets in Manin's Conjecture.
\end{abstract}

\maketitle

\section{Introduction}

Let $X$ be a smooth projective Fano variety defined over a number field $F$ and $\mathcal L = (L, \|\cdot\|)$ be an adelically metrized divisor on $X$. Then one can associate a real valued function on the set of rational points:
\[
\mathsf H_{\mathcal L} : X(F) \rightarrow \mathbb R_{>0}
\]
to a triple $(F, X, \mathcal L)$. This is called the height function associated to $(F, X, \mathcal L)$.  (See \cite{CLT10} for the definitions of adelic metrizations and height functions.) 
When $L$ is ample, the height function enjoys the Northcott property, i.e., for any real number $T$ the set of rational points of height $\leq T$
\[
\{P \in X(F) | \mathsf H_{\mathcal L}(P) \leq T\}
\]
is a finite set. Thus one may define, for any subset $Q \subset X(F)$, the counting function
\[
N(Q, \mathcal L, T) = \# \{P \in Q | \mathsf H_{\mathcal L}(P) \leq T\}.
\]
Manin's Conjecture, originally formulated in \cite{BM}, predicts the asymptotic formula of $N(Q, \mathcal L, T)$ for an appropriate choice of $Q$ in terms of two birational invariants of $(X, L)$, denoted by $a(X, L), b(F, X, L)$. (See Section~\ref{subsec:invariants} for the definitions of these invariants.) 

To state Manin's Conjecture, we need to introduce the notion of {\it thin sets}: 
\begin{defi}[Thin sets]
Let $X$ be a variety defined over $F$. {\it A thin map} is a generically finite morphism to the image from a variety defined over $F$ such that if it is dominant, then it is not birational. {\it A thin set} is a finite union of subsets of $X(F)$, which are of the form of $f(Y(F))$ where $f$ is a thin map. 
\end{defi}
Here is a version of Manin's Conjecture using the notion of thin sets:
\begin{conj}[Manin's Conjecture]
\label{conj:Manin}
Let $F$ be a number field and $X$ be a smooth projective geometrically irreducible and geometrically rationally connected variety defined over $F$. Let $\mathcal L$ be an adelically metrized big and nef $\mathbb Q$-divisor on $X$.

Suppose that $X(F)$ is not thin. Then there exists $Z \subset X(F)$, which is contained in some thin subset of $X(F)$, such that we have
\[
N(X(F)\setminus Z, \mathcal L, T) \sim c(F, Z, \mathcal L)T^{a(X, L)}(\log T)^{b(F, X, L)-1}
\]
where $c(F, Z, \mathcal L)$ is Peyre's constant, introduced in \cite{Peyre} and \cite{BT}.
\end{conj}

Originally Manin's Conjecture predicted that the set $Z$, which is called an exceptional set, is contained in a proper closed subset. The closed set version of Manin's Conjecture is known for many examples such as generalized flag varieties, toric varieties, low degree hypersurfaces, and so on. (See \cite{FMT89}, \cite{BT-0}, \cite{BT-general}, and \cite{Bir61}.) However, there are also many counterexamples to this version of Manin's Conjecture first found by Batyrev and Tschinkel in \cite{BT-cubic}. (See, e.g., \cite{LeRudulier}, \cite{BL16}, and \cite{BHB18} for other counterexamples to the closed set version of Manin's Conjecture.) In \cite{Peyre03}, Peyre first predicted that an exceptional set in Manin's Conjecture should be contained in a thin subset and so far there is no counterexample to this version of Manin's Conjecture. Note that for a rational variety, the set of rational points is not thin, and it is expected that for a smooth geometrically rationally connected variety, the set of rational points is not thin after taking some finite extension. So we believe that the assumption of Conjecture~\ref{conj:Manin} is not strict.

In a series of papers \cite{HTT15}, \cite{LTT14}, \cite{HJ16}, \cite{LTDuke}, \cite{Sen17}, \cite{LST18}, and \cite{LTsurvey}, Hassett, Tschinkel, Lehmann, Hacon, Jiang, Sengupta and the author study the geometry of exceptional sets in Manin's Conjecture. In \cite{LST18}, the author with Lehmann and Sengupta proposes the geometric construction of exceptional sets and proves that it is indeed contained in a thin set using the minimal model program \cite{BCHM} and the boundedness of singular Fano varieties proved in \cite{birkar16} and \cite{birkar16b} as well as the Hilbert Irreducibility Theorem in \cite{Serre}. One of main theorems of \cite{LST18} is the following theorem which shows the geometric consistency of Manin's Conjecture:

\begin{theo}{\cite[Theorem 1.2]{LST18}} \label{theo: maintheorem}
Let $X$ be a smooth geometrically uniruled projective variety over a number field $F$ and let $L$ be a big and nef divisor on $X$.  As we vary over all $F$-thin maps $f: Y \to X$ with $Y$ smooth, projective, and geometrically integral such that
\[
(a(X, L), b(F, X, L)) < (a(Y, f^*L), b(F, Y, f^*L))
\] 
in the lexicographic order, the points
\begin{equation*}
\bigcup_{f} f(Y(F))
\end{equation*}
are contained in a thin subset of $X(F)$.
\end{theo}

In this note we recall the construction of exceptional sets in \cite{LST18} and discuss some aspects of a proof of a weaker version of the above theorem, which was originally proved in \cite{LTDuke}.

\bigskip

\noindent
{\bf Acknowledgments.}
The author would like to thank Brendan Hassett, Brian Lehmann, Akash Sengupta, and Yuri Tschinkel for collaborations helping to shape his perspective on the geometry of Manin's Conjecture.
In particular the author would like to thank his advisor Yuri Tschinkel for introducing him this subject and Brian Lehmann for multiple fruitful collaborations.
The author would also like to thank the organizers of Kinosaki Algebraic Geometry Symposium 2018, Atsushi Kanazawa, Hokuta Uehara, and Kiwamu Watanabe, for their kind invitation and the opportunity to speak at their conference.
The author thanks Brian Lehmann for comments on an earlier draft of this paper.
Sho Tanimoto is partially supported by MEXT Japan, Leading Initiative for Excellent Young Researchers (LEADER).

\section{Background}

Let $F$ be a field of characteristic $0$. A variety defined over $F$ is an integral separated scheme of finite type over $F$. 

Let $X$ be a projective variety defined over $F$. Let $N^1(X)$ be the real N\'eron-Severi space of $X$, i.e., the space of $\mathbb R$-Cartier divisors modulo numerical equivalence. 
Let $N_1(X)$ be the dual of $N^1(X)$ which can be considered as the space of $1$-cycles modulo numerical equivalence.
We denote the pseudo effective cones by 
\[
\overline{\mathrm{Eff}}^1(X) \subset N^1(X), \quad \overline{\mathrm{Eff}}_1(X) \subset N_1(X),
\]
which are the closures of effective $\mathbb R$-cycles in the real vector spaces.
Dually we have the nef cones
\[
\mathrm{Nef}^1(X) \subset N^1(X), \quad \mathrm{Nef}_1(X) \subset N_1(X).
\]
The nef cone of divisors $\mathrm{Nef}^1(X)$ is dual to the pseudo-effective cone of curves $\overline{\mathrm{Eff}}_1(X)$ and the nef cone of curves $\mathrm{Nef}_1(X)$ is dual to the pseudo-effective cone of divisors $\overline{\mathrm{Eff}}^1(X)$.

\subsection{Geometric invariants}
\label{subsec:invariants}

Here we recall the definitions of $a(X, L), b(F, X, L)$ appearing in Manin's Conjecture.
\begin{defi}[Fujita invariants]
Let $X$ be a smooth projective variety defined over $F$ and $L$ be a big and nef $\mathbb Q$-divisor on $X$. {\it The Fujita invariant} (or {\it the $a$-invariant}) of $(X, L)$ is defined by
\[
a(X, L) = \min \{t \in \mathbb R | tL + K_X \in \overline{\mathrm{Eff}}^1(X) \}.
\]
As the notation suggests $a(X, L)$ is invariant under the base change of the ground field. By \cite{BDPP}, $a(X, L) > 0$ if and only if $X$ is geometrically uniruled. When $L$ is not big, we formally set $a(X, L) = + \infty$.

When $X$ is singular we take a resolution $\beta : \widetilde{X} \rightarrow X$ and define the Fujita invariant by
\[
a(X, L) := a(\widetilde{X}, \beta^*L).
\]
This is well-defined because $a(X, L)$ is a birational invariant. See \cite[Proposition 2.7]{HTT15}.
\end{defi}

Regarding the $a$-invariants, we frequently use the following notion:

\begin{defi}
Let $X$ be a projective variety defined over $F$ and $L$ be a big and nef $\mathbb Q$-Cartier divisor on $X$. We say that the pair $(X, L)$ is adjoint rigid if there exists a smooth resolution $\beta: \widetilde{X} \rightarrow X$ such that $a(X, L)\beta^*L + K_{\widetilde{X}}$ is rigid, i.e., it has Iitaka dimension $0$. Since $a(X, L)$ is a birational invariant, this property does not depend on the choice of a resolution $\beta: \widetilde{X} \rightarrow X$.
\end{defi}

We also use the following theorem implicitly through out the paper:

\begin{theo}{\cite{HMX13}}
Let $\mathcal X \rightarrow Y$ be a smooth family of geometrically uniruled projective varieties defined over $F$ and $L$ be $f$-big and $f$-nef $\mathbb Q$-divisor on $\mathcal X$. Then there exists a non-empty Zariski open subset $U \subset Y$ such that the function $U \rightarrow \mathbb R$ mapping $y \in U \mapsto a(X_y, L)$ is constant.
\end{theo}
\begin{proof}
For a proof, see \cite[Theorem 4.3]{LTDuke}.
\end{proof}

\begin{defi}[the $b$-invariants]
Let $X$ be a smooth geometrically uniruled projective variety defined over $F$ and $L$ be a big and nef $\mathbb Q$-divisor on $X$. {\it The $b$-invariant} of $(X, L)$ is defined by
\begin{align*}
b(F, X, L) = \textrm{the }& \text{codimension of the minimal supported face}\\
& \text{of $\overline{\mathrm{Eff}}^1(X)$ containing $K_{X} + a(X, L)L$.}
\end{align*}
This value $b(F, X, L)$ is not invariant under the base change as the Picard rank of a projective variety $X$ depends on the ground field.
Again we define the $b$-invariant even for singular varieties via a passage to a smooth model as the case of Fujita invariants. This is well-defined because the $b$-invariant is a birational invariant. See \cite[Proposition 2.10]{HTT15}.
\end{defi}

The most important case of the $a, b$-invariants is the case of Fano varieties and $L$ being the anticanonical divisor.
\begin{exam}
Let $X$ be a smooth projective Fano variety defined over $F$ and $L = -K_X$. Then we have
\[
a(X, L) = 1, \quad b(F, X, L) = \dim N^1(X).
\]
\end{exam}

Here is the first counterexample to the closed set version of Manin's Conjecture:
\begin{exam}{\cite{BT-cubic}}
Let $F$ be an arbitrary number field. Let $X \subset \mathbb P^3_x \times \mathbb P^3_y$ be the hypersurface defined by
\[
\sum_{i = 0}^3 x_i y_i^3 = 0.
\]
Then $X$ is a smooth projective Fano fivefold with the Picard rank $2$.
Let $L = -K_X$ with an adelic metrization. Then we have
\[
a(X, L) = 1, \quad b(F, X, L) = 2.
\]
Thus if the closed set version of Manin's Conjecture is true, then there exists a Zariski open subset $X^\circ \subset X$ such that 
\[
N(X^\circ(F), \mathcal L, T) \sim cT (\log T)
\]
for some $c >0$. On the other hand consider the cubic surface fibration $\pi_1 : X \rightarrow \mathbb P^3_x$. For any $P = (x_0: x_1: x_2:x_3) \in \mathbb P^3(F)$ such that $x_i \in (F^\times)^3$ for any $i$, the fiber $X_P$ satisfies that $a(X_p, L) = 1$, and $b(F, X_p, L) = 7$ if $\sqrt{-3} \in F$ and $b(F, X_p, L) = 4$ if $\sqrt{-3} \not\in F$. Batyrev and Tschinkel showed that, under the assumption $\sqrt{-3} \in F$, for any open subset $U \subset X_P$, we have
\[
N(U(F), \mathcal L, T) \gg T(\log T)^3.
\]
(\cite{FLS18} removes the condition of $\sqrt{-3} \in F$.)
Thus this is a contradiction and $(X, L)$ cannot satisfy the closed set version of Manin's Conjecture.
The closed set version of Manin's Conjecture is expected to be true for smooth cubic surfaces after removing the contribution from lines. (See \cite[Theorem 9.1]{LTsurvey}.) At the moment of writing this report, Manin's Conjecture is not known even for a single smooth cubic surface defined over a number field. Proving Conjecture~\ref{conj:Manin} for $(X, L)$ is out of reach at this moement. (Though the Fano fivefold defined by $\sum_i x_i y_i^2 = 0$ has been handled in \cite{BHB18}.)
\end{exam}

\subsection{Face contracting}

In this section we recall the notion of face contracting which is a key notion for the geometric construction of exceptional sets in \cite{LST18}. First we note the following lemma:
\begin{lemm}
Let $X$ be a smooth geometrically unirulded projective variety defined over a field $F$ of characteristic $0$ and $L$ be a big and nef $\mathbb Q$-divisor on $X$. Let $f :  Y \rightarrow X$ be a generically finite dominant morphism from a projective variety. Then we have
\[
a(Y, f^*L) \leq a(X, L).
\]
\end{lemm}

\begin{proof}
After applying a resolution we may assume that $Y$ is smooth. Then by the ramification formula there exists an effective divisor $R \geq 0$ on $Y$ such that
\[
K_Y  = f^*K_X +R.
\]
Now we have
\[
a(X, L)f^*L + K_Y = f^*(a(X, L)L + K_X) + R
\]
which is pseudo-effective. Thus our assertion follows.
\end{proof}

\begin{defi}
Let $X$ be a smooth geometrically uniruled projective variety defined over a field $F$ and $L$ be a big and nef $\mathbb Q$-divisor on $X$. We define the supported face associated to $(X, L)$ by
\[
F(X, L) = \mathrm{Nef}_1(X) \cap \{ (a(X, L)L + K_X) = 0\}.
\]
Note that we have $\dim F(X, L) = b(F, X, L)$.

Let $f: Y \rightarrow X$ be a generically finite dominant morphism from a smooth projective variety. A cover $f$ is called {\it an $a$-cover} if $a(Y, f^*L) = a(X, L)$.

An $a$-cover $f :  Y \rightarrow X$ is called {\it face contracting} if the induced map
\[
f_* : F(Y, L) \rightarrow F(X, L)
\]
is not injective. If $b(F, X, L) < b(F, Y, f^*L)$ is true, then $f$ is automatically face contracting. However the converse is not true as \cite[Example 3.7]{LT17} shows.
\end{defi}

\section{The construction of exceptional sets and the main theorem}

In this section, we introduce the construction of a conjectural exceptional set from \cite{LST18} and discuss main theorems in this paper. We work over a number field $F$ in this section.

Let $X$ be a geometrically uniruled and geometrically integral smooth projective variety defined over $F$ and $L$ be a big and nef $\mathbb Q$-divisor on $X$. For simplicity we assume that $(X, L)$ is adjoint rigid, i.e., $a(X, L)L + K_X$ has Iitaka dimension $0$.
This condition ensures that $X$ is geometrically rationally connected by \cite{HM07}.
We denote the augmented base locus of $L$ by $\bold B_+(L)$: this is necessary a proper closed subset. (See \cite{Nak04} for its definition and basic properties.)
We set $Z_0$ to be the set of rational points on $\bold B_+(L)$.

Next as $f: Y \rightarrow X$ varies all $F$-thin maps such that
\begin{itemize}
\item $Y$ is geometrically integral, smooth and projective;
\item $\dim Y < \dim X$ or $a(Y, f^*L)f^*L + K_Y$ has positive Iitaka dimension;
\item and we have the inequality
\[
(a(X, L), b(F, X, L)) \leq (a(Y, f^*L), b(F, Y, f^*L))
\]
in the lexicographic order, 
\end{itemize}
we define the set $Z_1 \subset X(F)$ by
\[
Z_1 = \bigcup_f f(Y(F)).
\]

Finally as $f: Y \rightarrow X$ varies all $F$-thin maps such that
\begin{itemize}
\item $Y$ is geometrically integral, smooth and projective and $f$ is dominant;
\item $a(Y, f^*L)f^*L + K_Y$ has Iitaka dimension $0$;
\item we have the inequality
\[
(a(X, L), b(F, X, L)) \leq (a(Y, f^*L), b(F, Y, f^*L))
\]
in the lexicographic order;
\item and $f$ is face contracting,
\end{itemize}
we define the set $Z_2 \subset X(F)$ by
\[
Z_2 = \bigcup_f f(Y(F)).
\]
In \cite{LST18}, we prove the following theorem using BAB conjecture:

\begin{theo}{\cite[Theorem 3.5]{LST18}}
The set $Z_0 \cup Z_1 \cup Z_2$ is contained in a thin subset of $X(F)$.
\end{theo}

\begin{rema}
In the construction of the set $Z_2$, it is important to insist that $f$ is face contracting. Otherwise, the the above theorem is no longer true. See \cite[Example 8.7]{LTDuke}.
\end{rema}
Thus it is natural to propose the following refinement of Manin's Conjecture:
\begin{conj}
In Conjecture~\ref{conj:Manin}, assuming $a(X, L)L + K_X$ is rigid,  we can take $Z$ to be $Z_0 \cup Z_1 \cup Z_2$.
\end{conj}

\section{A proof}

In this section, we explain a proof of the following weaker theorem which was originally proved in \cite{LTDuke}:
\begin{theo}{\cite[Theorem 1.5]{LTDuke}}
\label{theo:Duke}
Let $X$ be a geometrically uniruled and geometrically integral smooth projective variety defined over a number field $F$ and let $L$ be a big and nef $\mathbb Q$-divisor on $X$. Suppose that the geometric Picard rank and the arithmetic Picard rank of $X$ coincide, i.e., $\dim N^1(\overline{X}) = \dim N^1(X)$ and $a(X, L)L + K_X$ is rigid. As $Y$ varies all geometrically integral subvarieties $Y\subset X$ defined over $F$ such that 
\[
(a(X, L), b(X, L)) \leq (a(Y, L), b(F, Y, L))
\]
in the lexicographic order, the points 
\[
Z = \bigcup_Y Y(F) \subset X(F)
\]
are contained in a thin subset of $X(F)$.
\end{theo}

First we recall the following Lemma:
\begin{lemm}{\cite[Theorem 4.5]{LTT14} and \cite[Lemma 2.6]{LST18}}
Let $Y$ be a geometrically uniruled and geometrically integral smooth projective variety defined over a number field $F$ and let $L$ be a big and nef $\mathbb Q$-divisor on $Y$. Let $\pi : Y \dashrightarrow Z$ be the canonical fibration associated to $a(Y, L)L + K_Y$ where its existence is justified by \cite{BCHM}.
Then there exists a Zariski open subset $Z^\circ \subset Z$ such that for any $z \in Z^\circ(F)$, we have
\[
a(Y, L) = a(Y_z, L), \quad b(F, Y, L) \leq b(F, Y_z, L),
\]
where $Y_z$ is a fiber of $\pi$ at $z$.
\end{lemm}
\begin{proof}
Over an algebraically closed field, this is proved in \cite[Theorem 4.5]{LTT14}. Over a number field, this is explained in \cite[Lemma 2.6]{LST18}.
\end{proof}

According to this lemma, it is natural to look at the following set: as $Y$ varies all geometrically integral subvarieties $Y\subset X$ defined over $F$ such that $(Y, L)$ is adjoint rigid and 
\[
(a(X, L), b(X, L)) \leq (a(Y, L), b(F, Y, L))
\]
in the lexicographic order, we define the set
\[
Z' = \bigcup_Y Y(F) \subset X(F).
\]
It turns out that the set $Z$ is contained in $Z'$ up to a proper closed subset, so we will explain why $Z'$ is contained in a thin subset of $X(F)$. To this end, we need a special case of BAB conjecture where the full conjecture is settled by Birkar in \cite{birkar16} and \cite{birkar16b}:
\begin{theo}{\cite{birkar16} and \cite{birkar16b}}
Let $(X, \Delta)$ be a terminal pair defined over an algebraically closed field of characteristic $0$ such that $-(K_X + \Delta)$ is ample. Then there exists a constant $C = C(\dim X)$ which only depends on the dimension of $X$ such that
\[
\mathrm{Vol}(-K_X) \leq C.
\]
\end{theo}

Using this theorem, one can prove the boundedness of adjoint rigid subvarieties $Y$ with $a(Y, L) \geq a(X, L)$:
\begin{theo}
\label{theo:bounded}
Let $X$ be a uniruled smooth projective variety defined over an algebraically closed field of characteristic $0$ and let $L$ be a big and nef $\mathbb Q$-divisor on $X$. Then there exists a constant $C > 0$ such that for any adjoint rigid subvariety $(Y, L)$ such that $a(Y, L) \geq a(X, L)$ and $Y \not\subset \bold B_+(L)$, we have
\[
L^{\dim Y}. Y \leq C.
\]
Thus such $Y$'s form a bounded set in $\mathrm{Chow}(X)$.
\end{theo}
\begin{proof}
Pick a smooth resolution $\beta : \widetilde{Y} \rightarrow Y$. Our assumption implies that $a(Y, L)\beta^*L + K_{\widetilde{Y}}$ is rigid. After replacing $\beta^*L$ by an effective divisor $\widetilde{L}$ $\mathbb Q$-linearly equivalent to $\beta^*L$, we may assume that $(\widetilde{Y}, a(Y, L)\widetilde{L})$ is a terminal pair. Since $\widetilde{L}$ is big, one can run the MMP with respect to $a(Y, L)\widetilde{L} + K_{\widetilde{Y}}$ by \cite{BCHM} and one obtains a birational contraction map $\phi : \widetilde{Y} \dashrightarrow Y'$. After applying a resolution to $\widetilde{Y}$, one may assume that $\phi$ is a morphism. Since $a(Y, L)\beta^*L + K_{\widetilde{Y}}$ is rigid, we have
\[
a(Y, L)\phi_*\widetilde{L} + K_{Y'} \equiv 0.
\]
For an appropriate choice of $\widetilde{L}$, one can show that $(Y', a(Y, L)\phi_*\widetilde{L})$ is a terminal pair. Since $a(Y, L)\phi_*\widetilde{L}$ is big, we can write
\[
a(Y, L)\phi_*\widetilde{L}  =  A + E,
\]
where $A$ is an ample $\mathbb Q$-divisor and $E$ is an effective $\mathbb Q$-divisor. The above equality shows that
\[
\epsilon A \equiv -(K_{Y'} + E + (1-\epsilon)A)
\]
is ample and $(Y', E + (1-\epsilon)A)$ is a terminal pair for sufficiently small $\epsilon > 0$. Thus by BAB conjecture we have
\[
\mathrm{Vol}(\phi_*\widetilde{L}) \leq \frac{\mathrm{Vol}(-K_{Y'})}{a(Y, L)^{\dim Y}} \leq \frac{C}{a(X, L)^{\dim Y}}.
\]
By the negativity lemma one can show that $L^{\dim Y} = \mathrm{Vol}(\beta^*L) \leq \mathrm{Vol}(\phi_*\widetilde{L})$. Thus our assertion follows. The last statement follows from \cite[Lemma 4.7]{LTT14}.
\end{proof}

This enable us to prove closedness of the exceptional set for the $a$-invariant. First we record the following lemma:

\begin{lemm}{\cite[Proposition 4.1]{LTT14}}
\label{lemm:dominant}
Let $X$ be a uniruled smooth projective variety and $L$ a big and nef $\mathbb Q$-divisor on $X$. Let $\pi : \mathcal U \rightarrow W$ be a family of subvarieties on $X$ such that the evaluation map $s : \mathcal U \rightarrow X$ is dominant. Then for a general member $Y$ of $\pi$, we have $a(Y, L) \leq a(X, L)$.
\end{lemm}

Combining Theorem~\ref{theo:bounded} and Lemma~\ref{lemm:dominant}, one can achieve the following theorem:

\begin{theo}{\cite{HJ16} and \cite[Theorem 3.3]{LT17}}
\label{theo:HJ}
Let $X$ be a uniruled smooth projective variety and $L$ a big and nef $\mathbb Q$-divisor on $X$. 
Then the union of subvarieties $Y$ such that $a(Y, L) > a(X, L)$ is a proper closed subset.
\end{theo}

To prove Theorem~\ref{theo:Duke}, we need the following structural results on families of adjoint rigid subvarieties:

\begin{prop}{\cite[Proposition 4.14]{LTDuke}}
\label{prop:evaluation}
Let $X$ be a uniruled smooth projective variety and $L$ a big and nef $\mathbb Q$-divisor on $X$.
Suppose that we have a family of subvarieties $\pi : \mathcal U \rightarrow W$ such that (1) the evaluation map $s : \mathcal U \rightarrow X$ is dominant; (2) a general member $Y$ of $\pi$ satisfies $a(X, L) =a(Y, L)$ and $(Y, L)$ is adjoint rigid; (3) the induced rational map $W \dashrightarrow \mathrm{Chow}(X)$ is generically finite. Then $s : \mathcal U \rightarrow X$ is generically finite.
\end{prop}

\noindent
 {\it Proof of Theorem~\ref{theo:Duke}}: As we mentioned before, we need to show that $Z'$ is contained in a thin set. By Theorem~\ref{theo:HJ}, we only need to consider adjoint rigid subvarieties $Y$ with $a(Y, L) = a(X, L)$. By Theorem~\ref{theo:bounded}, there are only finitely many families to consider. Obviously it suffices to consider dominant families. For such a family $\pi_i : \mathcal U_i \rightarrow W$ it follows from Proposition~\ref{prop:evaluation} that the evaluation map $s_i : \mathcal U_i \rightarrow X$ is generically finite. If the degree of this evaluation map is greater than $1$, then thinness of the contribution of subvarieties in $\mathcal U_i$ is clear. Thus we may assume that $s_i : \mathcal U_i \rightarrow X$ is birational. Then one can appeal to the folowing proposition which follows from Hilbert Irreducibility Theorem as \cite[Proposition 3.3.5]{Serre}:

\begin{prop}{\cite[Proposition 5.1]{LTDuke}}
Let $X$ be a geometrically uniruled and geometrically integral smooth projective variety defined over a number field $F$ and let $L$ be a big and nef $\mathbb Q$-divisor on $X$ such that $\dim N^1(X) = \dim N^1(\overline{X})$ and $a(X, L)L + K_X$ is rigid. Furthermore we  assume that we have an algebraic fiber space $f : X \rightarrow Y$. Let 
\[
Y^\circ = \{y \in Y | \text{ $X_y$ is geometrically integral and smooth }\}.
\]
Then the following set 
\[
\{y \in Y^\circ(F) | a(X, L) =  a(X_y, L), b(X, L) \leq b(F, X_y, L)\}
\]
is contained in a thin subset of $Y^\circ(F)$.
\end{prop}

\section{Le Rudulier's example}

In this section we discuss the example from \cite{LeRudulier}. Let $S = \mathbb P^1 \times \mathbb P^1$ defined over the field of rational numbers $\mathbb Q$.
We define
\[
X := \mathrm{Hilb}^{[2]}(S), \quad L = -K_X
\]
which is the Hilbert scheme of zero-dimensional length two schemes on $S$. Then $X$ is a smooth projective weak Fano variety and in particular $-K_X$ is big and nef. Since $L$ is the anticanonical class, we have
\[
a(X, L) = 1, \quad b(\mathbb Q, X, L) = \dim N^1(X) = 3.
\]

On the other hand consider $S \times S$ and its quotient $g : S \times S \rightarrow \mathrm{Sym}^{(2)}(S)$ by the symmetric involution. We have the Hilbert-Chow morphism $\phi : X \rightarrow \mathrm{Sym}^{(2)}(S)$ which is a crepant resolution of $\mathrm{Sym}^{(2)}(S)$. Let $\psi : W \rightarrow S \times S$ be the blow-up of $S\times S$ along the diagonal. Then $W$ admits a degree $2$ finite morphism $f : W \rightarrow X$ and we have
\[
a(W, f^*L) = a(W, -\psi^*g^*K_{\mathrm{Sym}^{(2)}(S)}) = 1, \quad b(\mathbb Q, W, f^*L) = b(\mathbb Q, W, -\psi^*g^*K_{\mathrm{Sym}^{(2)}(S)}) = 4.
\]
Thus we have
\[
(a(X, L), b(\mathbb Q, X, L)) < (a(W, f^*L), b(\mathbb Q, W, f^*L))
\]
in the lexicographic order and thus the closed set version of Manin's Conjecture cannot be true for $(X, L)$. 

In \cite{LeRudulier}, Le Rudulier showed Conjecture~\ref{conj:Manin} after removing a thin exceptional set. Here is the description of her exceptional set: let $E$ be the exceptional divisor of the Hilbert-Chow morphism $\phi$. Let $D_1$ be the divisor parameterizing all subschemes supported on some (not fixed) fiber of the first projection $\mathbb P^1 \times \mathbb P^1 \rightarrow \mathbb P^1$. Let $D_2$ be the analogous divisor for the second projection. Then her exceptional set is
\[
Z = f(W(\mathbb Q)) \cup D_1(\mathbb Q) \cup D_2(\mathbb Q) \cup E(\mathbb Q).
\]

On the other hand, the geometry of $a, b$-invariants for $X$ has been worked out in \cite[Section 9.3]{LTDuke}. First of all we have $Z_0 \subset \bold B_+(L) = E$, so we have $Z_0 \subset Z$. In \cite{LTDuke}, Lehmann and the author showed that (1) all subvarieties with higher $a$-invariants are contained in $E\cup D_1\cup D_2$; (2) the only thin maps $h : Y \rightarrow X$ such that the image is not contained in $E\cup D_1\cup D_2$, $(Y, h^*L)$ is adjoint rigid, $\dim Y < \dim X$, and $(a(X, L), b(\mathbb Q, X, L)) \leq (a(Y, h^*L), b(\mathbb Q, Y, h^*L))$ are the images of the fibers of one of the projections $W \rightarrow S \times S \rightarrow \mathbb P^1$.
These imply that $Z_1 \subset Z$. Next we know that the geometric fundamental group of $\pi_1(X \setminus E\cup D_1\cup D_2)$ is $\mathbb Z/2$. Thus it follows from \cite{Sen17} that $(W, f^*L)$ is the only cover with $a(W, f^*L) = a(X, L) = 1$ and $(W, f^*L)$ is adjoint rigid. Finally by arguing as \cite[Example 8.6]{LTDuke}, one can show that all non-trivial twists of $f : W \rightarrow X$ have $a, b$-invariants less than $a(X, L), b(\mathbb Q, X, L)$. All together these imply that $Z_2 = f(W(\mathbb Q)).$ Thus we conclude
\[
Z = Z_0 \cup Z_1 \cup Z_2.
\]

\nocite{*}
\bibliographystyle{alpha}
\bibliography{absurvey}

\def\cprime{$'$}
\begin{thebibliography}{BCHM10}

\bibitem[BCHM10]{BCHM}
C.~Birkar, P.~Cascini, Chr.~D. Hacon, and J.~M\textsuperscript{c}Kernan.
\newblock Existence of minimal models for varieties of log general type.
\newblock {\em J. Amer. Math. Soc.}, 23(2):405--468, 2010.

\bibitem[BDPP13]{BDPP}
S.~Boucksom, J.~P. Demailly, M.~Paun, and T.~Peternell.
\newblock The pseudo-effective cone of a compact {K}\"ahler manifold and
  varieties of negative {K}odaira dimension.
\newblock {\em J. Algebraic Geom.}, 22(2):201--248, 2013.

\bibitem[BHB18]{BHB18}
T.~D. Browning and R.~Heath-Brown.
\newblock Density of rational points on a quadric bundle in ${P}^3 \times
  {P}^3$.
\newblock preprint, arXiv:1805.10715, 2018.

\bibitem[Bir62]{Bir61}
B.~J. Birch.
\newblock Forms in many variables.
\newblock {\em Proc. Roy. Soc. Ser. A}, 265:245--263, 1961/1962.

\bibitem[Bir16a]{birkar16}
C.~Birkar.
\newblock Anti-pluricanonical systems on {F}ano varieties, 2016.
\newblock arXiv:1603.05765 [math.AG].

\bibitem[Bir16b]{birkar16b}
C.~Birkar.
\newblock Singularities of linear systems and boundedness of {F}ano varieties,
  2016.
\newblock arXiv:1609.05543 [math.AG].

\bibitem[BL17]{BL16}
T.~D. Browning and D.~Loughran.
\newblock Varieties with too many rational points.
\newblock {\em Math. Z.}, 285(3):1249--1267, 2017.

\bibitem[BM90]{BM}
V.~V. Batyrev and Yu.~I. Manin.
\newblock Sur le nombre des points rationnels de hauteur born\'e des
  vari\'et\'es alg\'ebriques.
\newblock {\em Math. Ann.}, 286(1-3):27--43, 1990.

\bibitem[BT96a]{BT-general}
V.~V. Batyrev and Y.~Tschinkel.
\newblock Height zeta functions of toric varieties.
\newblock {\em J. Math. Sci.}, 82(1):3220--3239, 1996.
\newblock Algebraic geometry, 5.

\bibitem[BT96b]{BT-cubic}
V.~V. Batyrev and Y.~Tschinkel.
\newblock Rational points on some {F}ano cubic bundles.
\newblock {\em C. R. Acad. Sci. Paris S\'er. I Math.}, 323(1):41--46, 1996.

\bibitem[BT98a]{BT-0}
V.~V. Batyrev and Y.~Tschinkel.
\newblock Manin's conjecture for toric varieties.
\newblock {\em J. Algebraic Geom.}, 7(1):15--53, 1998.

\bibitem[BT98b]{BT}
V.~V. Batyrev and Y.~Tschinkel.
\newblock Tamagawa numbers of polarized algebraic varieties.
\newblock {\em Ast\'erisque}, (251):299--340, 1998.
\newblock Nombre et r\'epartition de points de hauteur born\'ee (Paris, 1996).

\bibitem[CLT10]{CLT10}
Antoine Chambert-Loir and Yuri Tschinkel.
\newblock Igusa integrals and volume asymptotics in analytic and adelic
  geometry.
\newblock {\em Confluentes Math.}, 2(3):351--429, 2010.

\bibitem[FLS18]{FLS18}
C.~Frei, D.~Loughran, and E.~Sofos.
\newblock Rational points of bounded height on general conic bundle surfaces.
\newblock {\em Proc. Lond. Math. Soc.}, 2018.
\newblock to appear.

\bibitem[FMT89]{FMT89}
J.~Franke, Yu.~I. Manin, and Y.~Tschinkel.
\newblock Rational points of bounded height on {F}ano varieties.
\newblock {\em Invent. Math.}, 95(2):421--435, 1989.

\bibitem[HJ17]{HJ16}
Chr.~D. Hacon and C.~Jiang.
\newblock On {F}ujita invariants of subvarieties of a uniruled variety.
\newblock {\em Algebr. Geom.}, 4(3):304--310, 2017.

\bibitem[HM07]{HM07}
Chr.~D. Hacon and J.~M\textsuperscript{c}Kernan.
\newblock On {S}hokurov's rational connectedness conjecture.
\newblock {\em Duke Math. J.}, 138(1):119--136, 2007.

\bibitem[HMX13]{HMX13}
Chr.~D. Hacon, J.~McKernan, and C.~Xu.
\newblock On the birational automorphisms of varieties of general type.
\newblock {\em Ann. of Math. (2)}, 177(3):1077--1111, 2013.

\bibitem[HTT15]{HTT15}
B.~Hassett, S.~Tanimoto, and Y.~Tschinkel.
\newblock Balanced line bundles and equivariant compactifications of
  homogeneous spaces.
\newblock {\em Int. Math. Res. Not. IMRN}, (15):6375--6410, 2015.

\bibitem[LR14]{LeRudulier}
C.~Le~Rudulier.
\newblock Points alg{\'e}briques de hauteur born{\'e}e sur une surface, 2014.
\newblock http://cecile.lerudulier.fr/Articles/surfaces.pdf.

\bibitem[LST18]{LST18}
B.~Lehmann, A.~K. Sengupta, and S.~Tanimoto.
\newblock Geometric consistency of {M}anin's {C}onjecture.
\newblock submited, 2018.

\bibitem[LT17a]{LT17}
B.~Lehmann and S.~Tanimoto.
\newblock Geometric {M}anin's {C}onjecture and rational curves.
\newblock submitted, 2017.

\bibitem[LT17b]{LTDuke}
B.~Lehmann and S.~Tanimoto.
\newblock On the geometry of thin exceptional sets in {M}anin's conjecture.
\newblock {\em Duke Math. J.}, 166(15):2815--2869, 2017.

\bibitem[LT18]{LTsurvey}
B.~Lehmann and S.~Tanimoto.
\newblock On exceptional sets in {M}anin's {C}onjecture.
\newblock submitted, 2018.

\bibitem[LTT18]{LTT14}
B.~Lehmann, S.~Tanimoto, and Y.~Tschinkel.
\newblock Balanced line bundles on {F}ano varieties.
\newblock {\em {\it J. Reine Angew. Math.}}, 743:91--131, 2018.

\bibitem[Nak04]{Nak04}
Noboru Nakayama.
\newblock {\em Zariski-decomposition and abundance}, volume~14 of {\em MSJ
  Memoirs}.
\newblock Mathematical Society of Japan, Tokyo, 2004.

\bibitem[Pey95]{Peyre}
E.~Peyre.
\newblock Hauteurs et mesures de {T}amagawa sur les vari\'et\'es de {F}ano.
\newblock {\em Duke Math. J.}, 79(1):101--218, 1995.

\bibitem[Pey03]{Peyre03}
E.~Peyre.
\newblock Points de hauteur born\'ee, topologie ad\'elique et mesures de
  {T}amagawa.
\newblock {\em J. Th\'eor. Nombres Bordeaux}, 15(1):319--349, 2003.

\bibitem[Sen17]{Sen17}
A.~K. Sengupta.
\newblock Manin's conjecture and the {F}ujita invariant of finite covers.
\newblock arXiv:1712.07780, 2017.

\bibitem[Ser92]{Serre}
J.~P. Serre.
\newblock {\em Topics in {G}alois theory}, volume~1 of {\em Research Notes in
  Mathematics}.
\newblock Jones and Bartlett Publishers, Boston, MA, 1992.
\newblock Lecture notes prepared by Henri Damon [Henri Darmon], With a foreword
  by Darmon and the author.

\end{thebibliography}

\end{document}